\documentclass[letterpaper, 10 pt, conference]{ieeeconf}  

\IEEEoverridecommandlockouts                              

\usepackage{graphicx}
\usepackage{graphics} 
\usepackage{amsmath} 
\usepackage{amssymb}  
\usepackage{amsfonts}
\usepackage[ruled,vlined]{algorithm2e}
\usepackage{pseudocode}
\usepackage{tabu}
\usepackage{color}
\usepackage{cite}
\let\labelindent\relax
\usepackage{enumitem}
\usepackage{epsfig}
\usepackage{mathrsfs}
\usepackage{array}
\usepackage{amsfonts,booktabs}
\usepackage{lipsum,amsmath,multicol}
\usepackage{graphicx}
\usepackage{subfigure}
\usepackage{times}
\usepackage{pdfsync}
\usepackage{pgfplots}
\usepackage{pgfplotstable}
\usepackage[subnum]{cases}
\usepackage{filecontents}
\usepackage{multirow}
\usepackage{float}
\usepackage{setspace}
\usepackage{mathtools}

\newtheorem{theorem}{\bf Theorem}

\newtheorem{remark}{\bf Remark}
\newtheorem{definition}{\bf Definition}

\newtheorem{lemma}{\bf Lemma}

\newcommand{\beq}{\begin{equation}}
\newcommand{\eeq}{\end{equation}}
\newcommand{\beqa}{\begin{eqnarray}}
\newcommand{\eeqa}{\end{eqnarray}}

\newcommand{\paren}[1]{\left(#1\right)}

\newcommand{\abs}[1]{\left|#1\right|} 
\newcommand{\I}[1]{\ensuremath{\mathsf{1}{\left\{#1\right\}}}} 
\newcommand{\PRP}[1]{\ensuremath{\mathsf{Pr}\left(#1\right)}} 
\newcommand{\ES}[1]{\ensuremath{\mathbb{E}\left[#1 \right]}} 

\newcommand{\logp}[1]{\ensuremath{\log\paren{#1}}}

\title{\LARGE \bf
Linearly Solvable Mean-Field Road Traffic Games
}

\author{
 \and Takashi Tanaka$^{1}$ \and Ehsan Nekouei$^{2}$\and 
Karl Henrik Johansson$^{3}$ 
\thanks{
$^{1}$Department of Aerospace Engineering and Engineering Mechanics, University of Texas at Austin, TX, USA.
        {\tt\small ttanaka@utexas.edu}. 
$^{2}$School of Electrical Engineering and Computer Science, KTH Royal Institute of Technology, Stockholm, Sweden.
        {\tt\small nekouei@kth.se}. 
$^{3}$School of Electrical Engineering and Computer Science, KTH Royal Institute of Technology, Stockholm, Sweden.
       {\tt\small kallej@kth.se}.}
}

\begin{document}

\maketitle
\thispagestyle{empty}
\pagestyle{empty}

\begin{abstract}
We analyze the behavior of a large number of strategic drivers traveling over an urban traffic network using the mean-field game framework. 
We assume an incentive mechanism for congestion mitigation under which each driver selecting a particular route is charged a tax penalty that is affine in the logarithm of the number of agents selecting the same route. 
We show that the mean-field approximation of such a large-population dynamic game leads to the so-called linearly solvable Markov decision process, implying that an open-loop $\epsilon$-Nash equilibrium of the original  game can be found simply by solving a finite-dimensional linear system.
\end{abstract}
%
%
\section{Introduction}
A reasonable approach to obtain a macroscopic model of a road traffic network is to use a game-theoretic analysis (e.g., Wardrop \cite{wardrop1952some}) with assumptions that (i) the number of players involved in the game is large, (ii) each individual  player's impact on the network is infinitesimal, and (iii) players' identities are indistinguishable \cite{patriksson2015traffic}.
Dynamic games with such assumptions are broadly known as \emph{mean-field games} (MFG), and have been actively studied in recent years \cite{caines2013mean, lasry2007mean}. 
The central result in the MFG theory shows that the game theoretic equilibria (e.g., the Markov perfect equilibria) of the original large-population game can be well-approximated by the  solution (called \emph{mean-field equilibrium}, MFE) to the pair of the backward Hamilton-Jacobi-Bellman (HJB) equation and the forward Fokker-Planck-Kolmogorov (FPK) equation. 
This result has a significant impact in applications where solving the HJB-FPK system is computationally more tractable than analyzing the equilibria of the original game with large number of players.

Recently, the MFG theory has been applied to the analysis of  traffic systems.
In \cite{CLM2015}, the authors modeled the interaction between drivers on a straight road as a non-cooperative game, and characterized the MFE of this game using an HJB equation and a mass preservation equation.
In \cite{BZP2017}, the authors considered a continuous-time Markov chain to model the aggregate behavior of drivers on a traffic network. They proposed various routing schemes for balancing the traffic flow over the network.  MFG has been applied to pedestrian  crowd dynamics modeling in \cite{lachapelle2011mean,dogbe2010modeling}.

In this paper, we apply the MFG framework to the analysis of an urban transportation network in which individual drivers' dynamics are decoupled from each other, but their cost functions are coupled via the tax penalty/reward mechanism imposed by the Traffic System Operator (TSO). In our context, the tax mechanism provides drivers with incentives to route themselves in such a way that their collective behavior matches the desirable traffic flow pre-calculated by the TSO. 
In particular, we are interested in the log-population type of the tax mechanisms (i.e., the penalty imposed on an individual driver is affine in the logarithm of the number of drivers taking the same route).
Such a tax mechanism is notable, as it renders the mean-field approximation of the original large-population game \emph{linearly solvable} using the result of \cite{todorov2007linearly}.
This offers a tremendous computational advantage over the conventional HJB-FPK characterization of the MFE.
The purpose of this paper is to delineate this observation and to demonstrate its computational advantages using a numerical simulation.

\subsection{Mean-field game theory: Background}

The MFG theory has been introduced by the authors of \cite{caines2013mean, lasry2007mean} and has been applied to the analysis of large-population games appearing in engineering, economic and financial applications. The central subject in the MFG theory is the coupled HJB-FPK system, which has attracted much attention in mathematical and engineering research \cite{lasry2007mean,achdou2010mean}.
The MFE of linear quadratic stochastic games have been extensively studied in the literature, e.g.,  \cite{Huang2010}, \cite{HLW2016}, \cite{MT2017} and references therein. MFGs with non-linear cost function and/or non-linear dynamics are also studied in, e.g., \cite{Huang2012, CLM2015}. 
The MFE of a Markov decision game with a major agent and a large number of minor agents was studied in \cite{Huang2012}, where the players are only coupled via their cost functions. These results were used in \cite{WYTH2014} to design decentralized security defense decisions in a mobile ad hoc network.
The authors of \cite{TH2011} studied the MFE of a dynamic stochastic game wherein the  dynamic of each agent is described by a (non-linear) stochastic difference equation, and agents are coupled via both dynamics and cost functions.
The authors in \cite{BTB2016} considered a stochastic dynamic game in which the dynamics and cost function of each agent is affected by its individual disturbance term. They studied the existence of a robust (minimax) MFE. In \cite{ZTB2011}, the authors analyzed the MFE of a hybrid stochastic game in which the dynamics of agents are affected by continuous disturbance terms as well as random switching signals.
Risk-sensitive MFG was considered in  \cite{tembine2014risk}.
While continuous-time continuous-state models are commonly used in the references above, the MFG in discrete-time and/or discrete-state regime have been  considered in, e.g., 
\cite{jovanovic1988anonymous, weintraub2006oblivious,gomes2010discrete}.

\subsection{Contribution of this paper}
In this paper, we apply the MFG theory to study the strategic behavior of infinitesimal drivers traveling over an urban traffic network. We consider a dynamic stochastic game wherein, at each intersection, each driver randomly selects one of the outgoing links as her next destination according to her randomized policy.
The objective function of each driver consists of the cost of taking different links at different time-steps as well as a tax penalty/incentive term computed by the TSO. At a given time, the tax associated with a particular link depends on the log-population of drivers traveling on that link.
Our problem set up is motivated by the \emph{mechanism design} problem in which the TSO's objective is to keep the empirical distribution of agents over the links at each time as close as possible to a target distribution. 

As the main technical contribution of this paper, we prove that the MFE of the game described above is given by the solution to a linearly solvable MDP. We emphasize that the MFE in our setting can be computed by performing a sequence of matrix multiplications backward in time \emph{only once}, without any need of forward-in-time computations. This is a stark contrast to the standard MFG results where there is a need to solve a forward-backward HJB-FPK system, which is often a non-trivial task \cite{achdou2010mean}. To highlight this computational advantage, we restrict ourselves to the discrete-time, discrete-space setting. Our result is different from \cite{gomes2010discrete} although the entropy-like cost considered there is similar to the  Kullback-Leibler cost that appears in our analysis. The game considered in \cite{gomes2010discrete} involves only a fixed number of players (which can be thought of as routing policies at intersections rather than infinitesimal drivers) and no mean-field limit is considered.

\section{Problem Formulation}
\label{secprob}
Let the directed graph $\mathcal{G}=(\mathcal{V},\mathcal{E})$ (referred to as the \emph{traffic graph}) represent the topology of the underlying traffic network, where $\mathcal{V}=\{1, 2, ... , V\}$ is the set of nodes (intersections) and $\mathcal{E}=\{1, 2, ... , E \}$ is the set of directed edges (links). For each $i\in \mathcal{V}$, denote by $\mathcal{V}(i) \subseteq \mathcal{V}$ the set of intersections to which there is a directed link from the intersection $i$.   
Let $\mathcal{N}=\{1, 2, ... , N\}$ be the set of players (drivers) on the graph $\mathcal{G}$. 
At any given time step $t\in \mathcal{T}=\{0, 1, ... , T-1\}$, each player is located at an intersection.
The node  at which the $n$-th player is located at time step $t$ is denoted by $i_{n,t}\in \mathcal{V}$. At every time step, player $n$ selects an \emph{action} $j_{n,t}\in \mathcal{V}(i_{n,t})$, which represents her next destination.
By selecting $j_{n,t}$ at time $t$, the player $n$ moves to the node $j_{n,t}$ at time $t+1$ deterministically (i.e., $i_{n,t+1}=j_{n,t}$).
Let $P_0=\{P_0^i\}_{i \in\mathcal{V}}$ be a probability mass function over $\mathcal{V}$. 
At $t=0$, we assume $i_{n,0}$ for each $n \in\mathcal{N}$ is realized independently with probability distribution  $P_0$.

\subsection{Players' strategies}
Each player traverses on $\mathcal{G}$ according to her individual randomized policy. More precisely, for each $n\in\mathcal{N}$, $i\in\mathcal{V}$ and $t\in\mathcal{T}$, let 
$Q_{n,t}^i=\{Q_{n,t}^{ij}\}_{j\in \mathcal{V}(i)}$ be the decision policy of player $n$ at the intersection $i$ at time $t$, representing a probability distribution according to which she selects the next destination $j \in \mathcal{V}(i)$.
We consider the collection $Q_{n,t}=\{Q_{n,t}^i\}_{i\in \mathcal{V}}$ of such probability distributions as the strategy of player $n$ at time $t$.
Namely, for each $n\in \mathcal{N}$ and $t\in \mathcal{T}$, we have $Q_{n,t}\in \mathcal{Q}$, where
\[
\mathcal{Q}=\left\{\{Q^{ij}\}_{i\in\mathcal{V}, j \in \mathcal{V}(i)}:\begin{array}{l}Q^{ij}\geq 0 \;\; \forall i\in \mathcal{V}, j\in\mathcal{V}(i) \\
\sum_{j\in\mathcal{V}(i)} Q^{ij}=1 \;\; \forall i\in \mathcal{V} \end{array}\right\}
\]
is the space of strategies. As clarified below, in this paper we consider a game with the open-loop information pattern \cite{basar1999dynamic}.

If the strategy $\{Q_{n,t}\}_{t\in\mathcal{T}}$ of player $n$ is fixed, then the probability distribution $P_{n,t}=\{P_{n,t}^i\}_{i\in \mathcal{V}}$ of her location at time $t$ is recursively computed by
\begin{equation}
\label{eqdyn}
P_{n,t+1}^j=\sum_i P_{n,t}^i Q_{n,t}^{ij} \;\; \forall t\in\mathcal{T}, j \in\mathcal{V}
\end{equation}
with the initial condition $P_{n,0}=P_0$.
If $(i_{n,t}, j_{n,t})$ is the location-action pair of player $n$ at time $t$, it has a joint distribution $P_{n,t}^i Q_{n,t}^{ij}$, and it is statistically independent of $(i_{m,t}, j_{m,t})$ with $m\neq n$.

\subsection{Action costs}
For each $i\in\mathcal{V}$, $j\in\mathcal{V}(i)$ and $t\in\mathcal{T}$, let $C_t^{ij}$ be a given constant that represents the cost (e.g., fuel cost) incurred to each agent who takes action $j$ at location $i$ at time $t$. We will also introduce the terminal cost $C_T^i$, $i\in\mathcal{V}$ for each player arriving at state $i$ at the final time step $t=T$.

\subsection{Tax mechanisms and incentives}
We assume that players are also subject to individual and time-varying tax penalties calculated by the TSO. Individual tax values depend not only on the players' locations and actions, but also on how the entire population is distributed over the traffic graph $\mathcal{G}$. Specifically, we consider the following log-population tax mechanism, where the tax charged to player $n$ taking action $j$ at location $i$ at time $t$ is
\begin{equation}
\label{eqtax}
\pi_{N,t,n}^{ij}=  \alpha  \left( \log \frac{K_{N,t}^{ij}}{K_{N,t}^i }-\log R_t^{ij} \right).
\end{equation}
In \eqref{eqtax}, $\alpha>0$ is a fixed constant. The parameter $R_t^{ij}> 0$ is also a fixed constant satisfying
$\sum_j R_t^{ij}=1$ for each $i$. $R_t^{ij}$ can be interpreted as the reference policy (state transition probability) designated by the TSO in advance. $K_{N,t}^i$ is the number of agents  (including player $n$) who are located at the intersection $i$ at time $t$, and $K_{N,t}^{ij}$ is the number of agents (including player $n$) who takes the action $j$ at the intersection $i$ at time $t$.
The tax rule \eqref{eqtax} indicates that agent $n$ receives a positive payment by taking action $j$ at location $i$ at time $t$ if $K_{N,t}^{ij}/K_{N,t}^{i} < R_t^{ij}$, while she is penalized by doing so if $K_{N,t}^{ij}/K_{N,t}^{i} > R_t^{ij}$. Since $K_{N,t}^{i}$ and $K_{N,t}^{ij}$ are random variables, $\pi_{N,t,n}^{ij}$ is also a random variable.
We assume that the TSO is able to observe $K_{N,t}^{i}$ and $K_{N,t}^{ij}$ at every time step and hence $\pi_{N,t,n}^{ij}$ is computable.\footnote{Whenever $\pi_{N,t,n}^{ij}$ is computed, we have both $K_{N,t}^{ij}\geq 1$ and $K_{N,t}^{i}\geq 1$ since at least player $n$ herself is counted. Hence \eqref{eqtax} is well-defined.}

Since player $n$'s probability of taking action $j$ at location $i$ at time step $t$ is given by $P_{n,t}^iQ_{n,t}^{ij}$, the total number $K_{N,t}^{ij}$ of such players follows the Poisson binomial distribution
\[
\text{Pr}(K_{N,t}^{ij}=k)=\sum_{A\in F_k}\prod_{n\in A} P_{n,t}^iQ_{n,t}^{ij}\!\!\prod_{n^c\in A^c}\!\!(1-P_{n^c,t}^iQ_{n^c, t}^{ij}).
\]
Here, $F_k$ is the set of all subsets of size $k$ that can be selected from $\mathcal{N}=\{1,2, ..., N\}$, and $A^c=F_k\backslash A$. Similarly, the distribution of $K_{N,t}^{i}$ is given by
\[
\text{Pr}(K_{N,t}^{i}=k)=\sum_{A\in F_k}\prod_{n\in A} P_{n,t}^i\!\!\prod_{n^c\in A^c}\!\!(1-P_{n^c,t}^i).
\]
Notice also that the conditional probability distribution of $K_{N,t}^{ij}$ given player $n$'s location-action pair $(i_{n,t},j_{n,t})=(i,j)$ is
\begin{align}
&\text{Pr}(K_{N,t}^{ij} =k+1 \mid i_{n,t}=i, j_{n,t}=j) \nonumber \\
&=\sum_{A\in F_k^{-n}}\prod_{m\in A} P_{m,t}^iQ_{m,t}^{ij}\!\!\prod_{m^c\in A^{-c}}\!\!(1-P_{m^c,t}^iQ_{m^c,t}^{ij}). \label{eqrhadcond}
\end{align}
Here, $F_k^{-n}$ is the set of all subsets of size $k$ that can be selected from $\mathcal{N}\backslash \{n\}$, and $A^{-c}=F_k^{-n}\backslash A$. Similarly, the conditional probability distribution of $K_{N,t}^{i}$ given $i_{n,t}=i$ is
\begin{align}
&\text{Pr}(K_{N,t}^{i} =k+1 \mid i_{n,t}=i) \nonumber \\
&=\sum_{A\in F_k^{-n}}\prod_{m\in A} P_{m,t}^i\!\prod_{m^c\in A^{-c}}\!\!(1-P_{m^c,t}^i). \label{eqrhadcond2}
\end{align}
Therefore, given the prior knowledge that the player $n$'s location-action pair at time $t$ is $(i,j)$, the expectation of her tax penalty $\pi_{N,n,t}^{ij}$  is 
\begin{align}
&\Pi_{N,n,t}^{ij}\triangleq \mathbb{E}\left[\pi_{N,n,t}^{ij} \mid  i_{n,t}=i, j_{n,t}=j \right] \nonumber \\
=&\sum_{k=0}^{N-1} \!\alpha\log\frac{k+1}{N}\!\!\!\!\sum_{A\in F_k^{-n}}\!\prod_{m\in A}\!\! P_{m,t}^iQ_{m,t}^{ij} \!\!\!\! \prod_{m^c\in A^c}\!\!\!(1\!-\!P_{m^c\!, t}^iQ_{m^c\!, t}^{ij}) \nonumber \\
&- \sum_{k=0}^{N-1} \!\alpha\log\frac{k+1}{N}\!\!\!\!\sum_{A\in F_k^{-n}}\!\prod_{m\in A}\!\! P_{m,t}^i \!\!\!\! \prod_{m^c\in A^c}\!\!\!(1\!-\!P_{m^c\!, t}^i) \nonumber \\
&- \alpha\log R_t^{ij}.
\label{eqtijgeneral}
\end{align}
Notice that the quantity \eqref{eqtijgeneral} depends on the strategies $Q_{-n}\triangleq \{Q_m\}_{m \neq n}$, but not on $Q_n$. In other words, $\pi_{N,n,t}^{ij}$ is a random variable whose distribution does not depend on player $n$'s own strategy. This fact will be used in Section~\ref{secmfe}.

\subsection{Road traffic game}
\label{sectrafficgame}
The $N$-player dynamic game considered in this paper is now formulated as follows.
\subsubsection{State dynamics}
We consider the probability distribution $P_{n,t}$ as the state of player $n$ at time $t$, and $Q_{n,t}$ as her control input.
Each individual's state dynamics is governed by \eqref{eqdyn}. Notice that different players' dynamics are decoupled.
\subsubsection{Cost functionals}
The $n$-th player's cost functional is given by
\begin{align*}
&J\left(\{Q_{n,t}\}_{t\in\mathcal{T}}, \{Q_{-n,t}\}_{t\in\mathcal{T}}\right) \\
&=\sum_{t=0}^{T-1} \sum_{i,j} P_{n,t}^i Q_{n,t}^{ij}\left(C_t^{ij}+\Pi_{N,n,t}^{ij}\right)+\sum_i P_{n,T}^i C_T^i.
\end{align*}
Notice that this quantity depends not only on the $n$-th player's own strategy $\{Q_{n,t}\}_{t\in\mathcal{T}}$ but also on the other players' strategies $\{Q_{-n,t}\}_{t\in\mathcal{T}}$ through the term $\Pi_{N,n,t}^{ij}$, whose precise expression is given by \eqref{eqtijgeneral}.

\subsubsection{Information pattern}
Throughout this paper, we restrict our analysis to the open-loop information pattern \cite{basar1999dynamic}.
More precisely, each player $n$ must fix a sequence of strategies $\{Q_{n,t}\}_{t\in\mathcal{T}}$ in advance based only on the public knowledge $\mathcal{G}, \alpha, \mathcal{N}, \mathcal{T}, R_t^{ij}, C_t^{ij}, C_T^i$ and $P_0$. Players are not allowed to update their strategies in real-time based on the observations $\{K_{N,t}^{ij}\}_{t\in\mathcal{T}}$.\footnote{In the future, we will consider a closed-loop implementation in which the open-loop optimization is performed repeatedly over the receding horizon.}

\subsubsection{Solution concept}
We introduce the following equilibrium concepts for the game described above.
\begin{definition}
The $N$-tuple of strategies $\{Q_{n,t}^{\text{NE}}\}_{n\in\mathcal{N}, t\in\mathcal{T}}$ is said to be an (open-loop) \emph{Nash equilibrium} if the following inequality holds for each $n\in \mathcal{N}$ and $\{Q_{n,t}\}_{t\in \mathcal{T}}, Q_{n,t} \in \mathcal{Q}$:
\[
J\left(\{Q_{n,t}\}_{t\in\mathcal{T}}, \{Q_{-n,t}^{\text{NE}}\}_{t\in\mathcal{T}}\right) \geq J\left(\{Q_{n,t}^{\text{NE}}\}_{t\in\mathcal{T}}, \{Q_{-n,t}^{\text{NE}}\}_{t\in\mathcal{T}}\right).
\vspace{1.5ex}
\]
\end{definition}
\begin{definition}
A Nash equilibrium $\{Q_{n,t}^{\text{NE}}\}_{n\in\mathcal{N}, t\in\mathcal{T}}$ is said to be \emph{symmetric} if 
\[\{Q_{1,t}^{\text{NE}}\}_{t\in\mathcal{T}}=\{Q_{2,t}^{\text{NE}}\}_{t\in\mathcal{T}}=\cdots=\{Q_{N,t}^{\text{NE}}\}_{t\in\mathcal{T}}.
\vspace{1.5ex}
\]
\end{definition}
\begin{remark}
The $N$-player game described above is a \emph{symmetric game} in the sense of \cite{cheng2004notes}. Thus, \cite[Theorem~3]{cheng2004notes} is applicable to show that it has a symmetric equilibrium.
\end{remark}
\begin{definition}
A set of strategies $\{Q_{n,t}^{\text{MFE}}\}_{n\in\mathcal{N}, t\in\mathcal{T}}$ is said to be an \emph{MFE} if the following conditions are satisfied.
\begin{itemize}
\item[(a)] It is symmetric, i.e.,  
\[\{Q_{1,t}^{\text{MFE}}\}_{t\in\mathcal{T}}=\{Q_{2,t}^{\text{MFE}}\}_{t\in\mathcal{T}}=\cdots=\{Q_{N,t}^{\text{MFE}}\}_{t\in\mathcal{T}}.\]
\item[(b)] There exists a sequence $\{\epsilon_N\}$ satisfying $\epsilon_N \searrow 0$ as $N\rightarrow \infty$ such that for each $n \in \mathcal{N}=\{1, 2, ... , N\}$ and for each $\{Q_{n,t}\}_{t\in\mathcal{T}}$ with $Q_{n,t} \in \mathcal{Q}$,
\begin{align*}
&J\left(\{Q_{n,t}\}_{t\in\mathcal{T}}, \{Q_{-n,t}^{\text{MFE}}\}_{t\in\mathcal{T}}\right) + \epsilon_N \\
&\hspace{6ex}\geq J\left(\{Q_{n,t}^{\text{MFE}}\}_{t\in\mathcal{T}}, \{Q_{-n,t}^{\text{MFE}}\}_{t\in\mathcal{T}}\right).
\end{align*}
\vspace{1ex}
\end{itemize}
\end{definition}

\section{Linearly Solvable MDPs}
\label{secoptcontrol}
In this section, we introduce an auxiliary optimal control problem that is closely related to the road traffic game introduced in the previous section. The result in this section will serve as a tool to find an MFE in the road traffic game described above.
The main emphasis in this section is that the introduced auxiliary optimal control problem belongs to the class of \emph{linearly-solvable MDPs} \cite{todorov2007linearly,dvijotham2011unified}. This fact provides a tremendous advantage in the computation of mean-field equilibria in the road traffic game. 

For each $t=0, ... , T$, let $P_t$ be the probability distribution over the vertices $\mathcal{V}$ that evolves according to
\[
P_{t+1}^j=\sum_i P_t^i Q_t^{ij} \;\; \forall  j\in\mathcal{V}
\]
with the initial state $P_0$.
We consider $P_t$ as the state of the dynamics and $Q_t$ as the control action in the optimal control problem below.
We assume $C_t^{ij}$, $R_t^{ij}$ for each $t\in\mathcal{T}, i\in\mathcal{I}, j\in\mathcal{I}$, $C_T^i$ for $i\in\mathcal{I}$, and $\alpha$ are given positive constants.
The $T$-step optimal control problem of our interest is:
\begin{equation}
\min_{\{Q_t\}_{t\in\mathcal{T}}}\sum_{t=0}^{T-1}\sum_{i,j}P_t^i Q_t^{ij}\left(C_t^{ij}+\alpha \log\frac{ Q_t^{ij}}{R_t^{ij}} \right)+\sum_i P_T^i C_T^i.
\label{eqoptcontrol2}
\end{equation}
Notice that the logarithmic term in \eqref{eqoptcontrol2} can be written as the Kullback--Leibler divergence from the reference policy $R_t^{ij}$ to the selected policy $Q_t^{ij}$.
For each $t=0, 1, ... , T$, introduce the value function:
\begin{align*}
V_{t}(P_{t}) & \triangleq \\
\min_{\{Q_{\tau}\}_{\tau=t}^{T-1}} &\!\sum_{\tau=t}^{T-1} \sum_{i,j} P_{\tau}^i Q_{\tau}^{ij}\!\left(C_\tau^{ij}+\alpha\log\frac{Q_\tau^{ij}}{R_\tau^{ij}} \right)+\sum_i P_{T}^i C_T^i
\end{align*}
and the associated Bellman equation
\begin{align}
&V_{t}(P_{t}) = \nonumber \\
&\min_{Q_t} \left\{\sum_{i,j} P_{\tau}^i Q_{\tau}^{ij}\!\left(C_\tau^{ij}+\alpha\log\frac{Q_\tau^{ij}}{R_\tau^{ij}} \right)\!+\!V_{t+1}(P_{t+1})\right\}. \label{eqbellmankl}
\end{align}
 The next theorem states that the optimal control problem \eqref{eqoptcontrol2} is \emph{linearly} solvable \cite{todorov2007linearly}.
\begin{theorem}
\label{theo1}
Let $\{\phi_t\}_{t\in\mathcal{T}}$ be the sequence of $V$-dimensional vectors defined by the backward recursion 
\begin{equation}
\label{eqphi}
\phi_t^i=\sum_j R_t^{ij} \exp \left(-\frac{C_t^{ij}}{\alpha}\right)\phi_{t+1}^j \;\; \forall i\in\mathcal{V}=\{1, 2, ... , V\}
\end{equation}
with the terminal condition $\phi_T^i=\exp(-C_T^i/\alpha) \; \forall i$.
Then, for each $t\in\mathcal{T}$ and $P_t$, the value function can be written as
\begin{equation}
V_t(P_t)=-\alpha\sum_i P_t^i \log \phi_t^i. \label{eqvt}
\end{equation}
Moreover, the optimal policy for \eqref{eqoptcontrol2} is given by
\begin{equation}
Q_t^{ij*}=\frac{\phi_{t+1}^j}{\phi_t^i}R_t^{ij}\exp\left(-\frac{C_t^{ij}}{\alpha}\right). \label{eqoptq}
\end{equation}
\end{theorem}
\begin{proof}
Due to the choice of the terminal condition $\phi_T^i=\exp(-C_T^i/\alpha)$, notice that
\[
V_T(P_T)=\sum_iP_T^i C_T^i = -\alpha\sum_i P_T^i \log \phi_T^i.
\]
Thus, \eqref{eqvt} holds for $t=T$. To complete the proof by backward induction, assume that
\[
V_{t+1}(P_{t+1})=-\alpha\sum_i P_{t+1}^i \log \phi_{t+1}^i
\]
holds for some $0\leq t \leq T-1$. Then, due to the Bellman equation \eqref{eqbellmankl}, we have
\[
V_{t}(P_{t}) =\min_{Q_t} \left\{ \sum_{i,j} P_{\tau}^i Q_{\tau}^{ij}\left(\rho_\tau^{ij}+\alpha\log\frac{Q_\tau^{ij}}{R_\tau^{ij}} \right)\right\}
\]
where $\rho_t^{ij}=C_t^{ij}-\alpha\log \phi_{t+1}^j$ is a constant. It is elementary to show that the minimum is attained by 
\[
Q_t^{ij*}=\frac{R_t^{ij}}{\phi_t}\exp\left(-\frac{\rho_t^{ij}}{\alpha}\right)
\]
from which \eqref{eqoptq} follows. By substitution, the optimal value is shown to be
\[
V_t(P_t)=-\alpha \sum_i P_t^i \log \phi_t^i.
\]
This completes the induction proof.
\end{proof}

We stress that \eqref{eqphi} is linear in $\phi$ and can be computed by matrix multiplications backward in time.

\section{Mean Field Equilibrium}
\label{secmfe}
Let $\{Q_t^*\}_{t\in\mathcal{T}}$ be the control policy obtained in 
 \eqref{eqoptq}, and let $\{P_t^*\}_{t\in\mathcal{T}}$ be defined recursively by
 \[
P_{t+1}^{j*}=\sum_i P_t^{i*} Q_t^{ij*} \;\; \forall  j\in\mathcal{V}.
\]
In this section, we consider the situation in which all players other than $n$ adopt the strategy $\{Q_t^*\}_{t\in\mathcal{T}}$, and then analyze player $n$'s best response.
For each $t\in\mathcal{T}$, the probability that  the $m$-th player ($m \neq n$) is located at $i$ is $P_t^{i*}$.
As before, define $\Pi_{N,n,t}^{ij}$ as the expected value of the tax penalty charged on player $n$ at time $t$ when she takes action $j$ at location $i$. Since players' dynamics over the traffic graph are decoupled, and $\Pi_{N,n,t}^{ij}$ is computed by the population excluding player $n$, $\Pi_{N,n,t}^{ij}$ does not depend on player $n$'s strategy.
Therefore, the best response by player $n$ is characterized by the solution to the following optimal control problem:
\begin{equation}
\label{eqprob_n}
\min_{\{Q_{n,t}\}_{t\in\mathcal{T}}} \sum_{t=0}^{T-1} \sum_{i,j} P_{n,t}^i Q_{n,t}^{ij}\left(C_t^{ij}+\Pi_{N,n,t}^{ij}\right)+\sum_i P_{n,T}^i C_T^i,
\end{equation}
where $\Pi_{N,n,t}^{ij}$ can be considered as a fixed constant.
To evaluate $\Pi_{N,n,t}^{ij}$ when all players other than player $n$ takes the same strategy (i.e., $Q_{m,t}=Q_t^*$ for $m\neq n$), notice that the conditional distributions of $K_{N,t}^{ij}$ and $K_{N,t}^{i}$ given $(i_{n,t}, j_{n,t})=(i,j)$, provided by \eqref{eqrhadcond} and \eqref{eqrhadcond2}, simplify to the binomial distributions
\begin{align*}
&\text{Pr}(K^{ij}_{N,t}=k+1|i_{n,t}=i, j_{n,t}=j)\\
&={N\!-\!1 \choose k}(P_t^{i*}Q_t^{ij*})^k (1-P_t^{i*}Q_t^{ij*})^{N-1-k}\\
&\text{Pr}(K^{i}_{N,t}=k+1|i_{n,t}=i)\\
&={N\!-\!1 \choose k}(P_t^{i*})^k (1-P_t^{i*})^{N-1-k}.
\end{align*}
Thus, the expression \eqref{eqtijgeneral} simplifies to
\begin{align}
&\Pi_{N,n,t}^{ij}=\mathbb{E}\left[\pi_{N,n,t}^{ij} \mid  i_n=i, j_n=j \right] \nonumber \\
&=\!\sum_{k=0}^{N-1}\!\alpha \log\frac{k+1}{N}{N\!-\!1 \choose k}(P_t^{i*}Q_t^{ij*})^k (1\!-\!P_t^{i*}Q_t^{ij*})^{N-1-k} \nonumber \\
&\;\;\;-\!\sum_{k=0}^{N-1}\!\alpha \log\frac{k+1}{N}{N\!-\!1 \choose k}(P_t^{i*})^k (1\!-\!P_t^{i*})^{N-1-k} \nonumber \\
&\;\;\;- \alpha \log R_t^{ij}
 \label{eqtjieq}
\end{align}

\subsection{Optimal solution to \eqref{eqprob_n} when $N \rightarrow \infty$}

Next, we study the asymptotic limit of $\Pi_{N,n,t}^{ij}$ as $N\rightarrow \infty$.
\begin{lemma}
\label{lemlimpi}
Let $\Pi_{N,n,t}^{ij}$ be defined by \eqref{eqtjieq}. If $P_t^{i*}Q_t^{ij*}>0$, then
\[
\lim_{N\rightarrow \infty} \Pi_{N,n,t}^{ij}=\alpha \log \frac{Q_t^{ij*}}{R_t^{ij}}.  
\]
\end{lemma}
\begin{proof}
See Appendix~\ref{app1}.
\end{proof}
Lemma~\ref{lemlimpi} implies that in the limit $N\rightarrow \infty$, the optimal control problem \eqref{eqprob_n} becomes
\begin{equation}
\min_{\{Q_{n,t}\}_{t\in\mathcal{T}}}  \sum_{t=0}^{T-1} \sum_{i,j} \!P_{n,t}^i Q_{n,t}^{ij}\!\left(\!C_t^{ij}\!+\!\alpha\log\frac{Q_t^{ij*}}{R_t^{ij}} \right) +\sum_i P_{n,T}^i C_T^i. \label{eqprob_n_inf}
\end{equation}
Notice that \eqref{eqprob_n_inf} is different from the auxiliary optimal control problem \eqref{eqoptcontrol2} studied in Section~\ref{secoptcontrol} in that the logarithmic term in \eqref{eqprob_n_inf} is a fixed constant that does not depend on the control policy. 
Nevertheless, these two optimal control problems are closely related as we show below.
To solve \eqref{eqprob_n_inf}, we once again apply  dynamic programming. For each $P_{n,t}$, define the value function
\begin{align*}
&V_{n,t}(P_{n,t})  \triangleq \\
&\min_{\{Q_{n,\tau}\}_{\tau=t}^T} \!\sum_{\tau=t}^{T-1} \sum_{i,j} P_{n,\tau}^i Q_{n,\tau}^{ij}\!\left(C_\tau^{ij}\!+\!\alpha\log\frac{Q_\tau^{ij*}}{R_\tau^{ij}} \right)\!+\!\sum_i P_{n,T}^i C_T^i
\end{align*}
The value function satisfies the Bellman equation:
\begin{align}
&V_{n,t}(P_{n,t})= \nonumber \\
&\min_{Q_{n,t}} \left\{ \sum_{i,j}P_{n,t}^i Q_{n,t}^{ij}\!\left(\!C_t^{ij}\!+\!\alpha\log\frac{Q_t^{ij*}}{R_t^{ij}}\right) \!+\! V_{n,t+1}(P_{n,t+1})\right\}. \label{eqbellman3}
\end{align}
The next key lemma shows that $V_{n,t}(\cdot)$ coincide with the value function $V_t(\cdot)$ for \eqref{eqoptcontrol2}. It also shows an interesting property of the optimal control problem \eqref{eqprob_n_inf} that \emph{any feasible control policy is an optimal control policy}. 

\begin{lemma}
\label{lembestresponse}
Let $\{\phi_t\}_{t\in\mathcal{T}}$ be the sequence defined by \eqref{eqphi}.
\begin{itemize}
\item[(a)] For each $t\in\mathcal{T}$ and $P_{n,t}$, we have 
\[
V_{n,t}(P_{n,t})=-\alpha \sum_i P_{n,t}^i \log \phi_t^{i}.
\]
\item[(b)] An arbitrary sequence of control actions $\{Q_{n,t}\}_{t\in\mathcal{T}}$ with $Q_{n,t}\in\mathcal{Q}$ is an optimal solution to \eqref{eqprob_n_inf}.
\end{itemize}
\end{lemma}
\begin{proof}
(a). Proof is by backward induction. If $t=T$, the claim trivially holds due to the definition $V_{n,T}(P_T)=\sum_i P_{n,T}^i C_T^i$ and the fact that the terminal condition for \eqref{eqphi} is given by $\phi_T^{i}=\exp(-C_T^i/\alpha)$. Thus, for $0\leq t \leq T-1$, assume that
\[
V_{n,t+1}(P_{n,t+1})=-\alpha \sum_j P_{n,t+1}^j \log \phi_{t+1}^{j} 
\]
holds.
Using $\rho_t^{ij}=C_t^{ij}-\alpha \log \phi_{t+1}^{j}$, the Bellman equation \eqref{eqbellman3} can be written as 
\begin{equation}
V_{n,t}(P_{n,t})=\min_{Q_{n,t}} \sum_{i,j}P_{n,t}^i Q_{n,t}^{ij}\left(\rho_t^{ij}+\alpha\log\frac{Q_t^{ij*}}{R_t^{ij}}\right). \label{eqbellman4}
\end{equation}
Substituting $Q_t^{ij*}$ obtained by \eqref{eqoptq} into \eqref{eqbellman4}, we have
\begin{subequations}
\label{eqvnt}
\begin{align}
V_{n,t}(P_{n,t}) &=\min_{Q_{n,t}}\sum_{i,j} P_{n,t}^i Q_{n,t}^{ij}\left(-\alpha \log \phi_t^{i}\right) \nonumber \\
&=\min_{Q_{n,t}} \sum_i P_{n,t}^i\left(-\alpha \log \phi_t^{i}\right) \underbrace{\sum\nolimits_j Q_{n,t}^{ij}}_{=1} \\
&=-\alpha \sum_i P_{n,t}^i \log \phi_t^{i}. \label{eqvntb}
\end{align}
\end{subequations}
This completes the proof.

(b). Since the final expression \eqref{eqvntb} does not depend on $Q_{n,t}$, any control action $Q_{n,t}\in \mathcal{Q}$ is a minimizer of the right hand side of the Bellman equation \eqref{eqbellman3}.
\end{proof}

Lemma~\ref{lembestresponse} (b) shows that an arbitrary policy is optimal in the optimal control problem \eqref{eqprob_n_inf}. This result is a reminiscent of the \emph{Wardrop's principle} \cite{wardrop1952some} (see also \cite{correa2011wardrop} and references therein), which states that travel costs are equal on all used routes at the   game-theoretic equilibrium among strategic and infinitesimal travelers. 

\subsection{Mean field equilibrium}
We are now ready to state the main result of this paper. 
The next theorem, together with Theorem~\ref{theo1}, provides a numerical method to compute an MFE of the road traffic game presented in Section~\ref{secprob}.
\begin{theorem}
\label{theomfe}
A symmetric strategy profile $Q_{n,t}^{ij}=Q_t^{ij*}$ for each $n\in\mathcal{N}, t\in \mathcal{T}$ and $i,j\in\mathcal{V}$, where $Q_t^{ij*}$ is obtained by \eqref{eqphi}--\eqref{eqoptq}, is an MFE of the road traffic game.
\end{theorem}
\begin{proof}
Let the policies $Q_{m,t}^{ij}=Q_t^{ij*}$ for $m \neq n$ be fixed. It is sufficient to show that there exists a sequence $\epsilon_N \searrow 0$ such that the cost of adopting a strategy $Q_{n,t}^{ij}=Q_t^{ij*}$  for player $n$ is no greater than $\epsilon_N $ plus the cost of adopting any other policy.
Since 
\[
\Pi_{N,n,t}^{ij}\rightarrow \alpha \log \frac{Q_t^{ij*}}{R_t^{ij}} \text{ as } N\rightarrow \infty,
\]
there exists a sequence $\delta_N \searrow 0$ such that
\[
\Pi_{N,n,t}^{ij}+\delta_N > \alpha \log \frac{Q_t^{ij*}}{R_t^{ij}} \;\; \forall i, j, t.
\]
Now, for all policy $\{Q_{n,t}\}_{t\in\mathcal{T}}$ of player $n$ and the induced distributions $P_{n,t}^j=\sum_i P_{n,t}^i Q_{n,t}^{ij}$, we have
\begin{align*}
&\sum_{t=0}^{T-1}\sum_{i,j} P_{n,t}^i Q_{n,t}^{ij} \left(C_t^{ij}+\Pi_{N,n,t}^{ij}\right) \\
&>\sum_{t=0}^{T-1}\sum_{i,j} P_{n,t}^i Q_{n,t}^{ij} \left(C_t^{ij}+\alpha \log \frac{Q_t^{ij*}}{R_t^{ij}}- \delta_N \right) \\
&=\sum_{t=0}^{T-1}\sum_{i,j} P_{n,t}^i Q_{n,t}^{ij} \left(C_t^{ij}+\alpha \log \frac{Q_t^{ij*}}{R_t^{ij}}\right)- T V^2 \delta_N\\
&\geq \min_{\{Q_{n,t}\}_{t\in\mathcal{T}}}\sum_{t=0}^{T-1}\sum_{i,j} P_{n,t}^i Q_{n,t}^{ij} \left(C_t^{ij}+\alpha \log \frac{Q_t^{ij*}}{R_t^{ij}}\right) \\
& \hspace{5ex} - T V^2 \delta_N. 
\end{align*}
Notice that the minimization in the last line is attained by adopting $Q_{n,t}^{ij}=Q_t^{ij*}$. Since $\epsilon_N \triangleq T V^2 \delta_N \searrow 0$, this completes the proof.
\end{proof}

\section{Numerical Illustration}
\label{secsimulation}

In this section, we illustrate the result of Theorem~\ref{theomfe} applied to a simple mean-field road traffic game over a traffic graph with 100 nodes (a grid world with obstacles) shown in Fig.~\ref{fig:simulation} and over time horizon $T=70$.
At $t=0$, the population is concentrated in the origin cell (indicated by ``O''). For each player $n$, the terminal cost is given by
$C_T^i=10\sqrt{\text{dist}(i, \text{D})}$, where $\text{dist}(i, \text{D})$ is the Manhattan distance between the player's final location $i$ and the destination cell (indicated by ``D''). For each time step $t$, the action cost for each player is given by
\[
C_t^{ij}=\begin{cases}
0 & \text{ if } j=i\\
1 & \text{ if } j\in\mathcal{V}(i)\\
100000 & \text{  if } j\not\in\mathcal{V}(i) \text{ or } j \text{ is an obstacle.}
\end{cases}
\]
where $\mathcal{V}(i)$ contains the north, east, south, and west neighborhood of the cell \#$i$.
As the reference distribution, we use $R_t^{ij}=1/|\mathcal{V}(i)|$ (uniform distribution) for each $i\in\mathcal{V}$ and $t\in \mathcal{T}$ to incentivize players to spread over the traffic graph.

For various values of $\alpha >0$, the backward formula \eqref{eqphi} is solved and the optimal policy is calculated by \eqref{eqoptq}.
If $\alpha$ is small (e.g., $\alpha=0.1$), it is expected that players will take the shortest path since the action cost is dominant compared to the tax cost \eqref{eqtax}. Numerical results confirm this intuition; three figures in the top row of Fig.~\ref{fig:simulation} show snapshots of the population distribution at time steps $t=20, 35$ and $50$, assuming all the players take the MFE policy obtained by \eqref{eqoptq}. In the bottom row, similar plots are generated with a larger $\alpha$ ($\alpha=1$). In this case, it can be seen that the equilibrium strategy will choose longer paths with higher probability to reduce congestion.

\begin{figure}[t]
    \centering
    \includegraphics[width=\columnwidth]{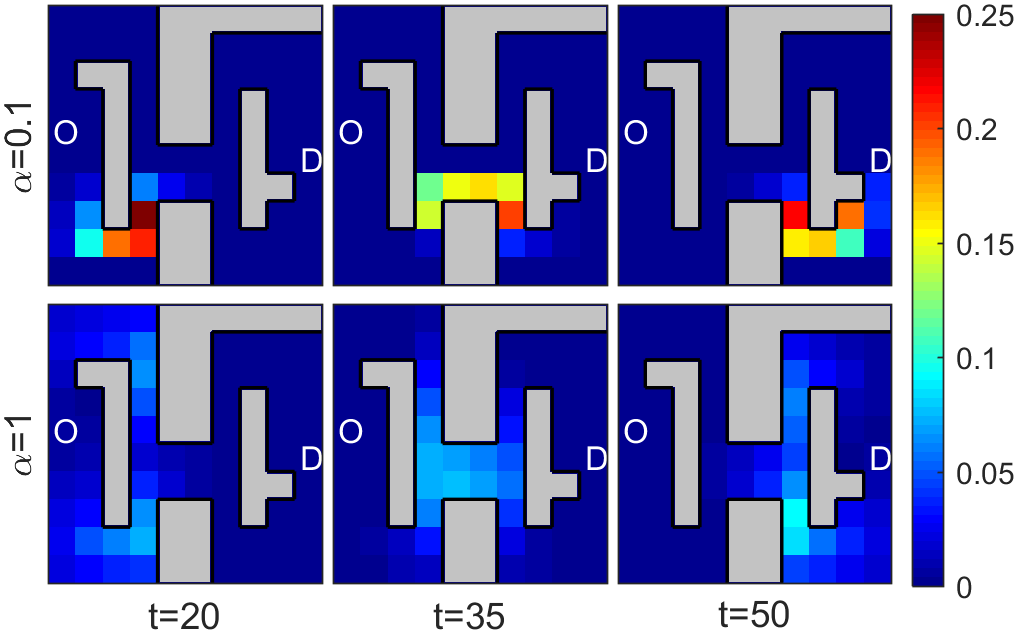}
    \caption{Simulation results for road traffic game at $t=20, 35, 50$ and for $\alpha=0.1$ and $1$.}
    \label{fig:simulation}
\end{figure}

\section{Conclusion and Future Work}
In this paper, we showed that the mean-field approximation of a large-population road traffic game under the log-population tax mechanism can be obtained by the linearly solvable MDP.
This result will serve as the basis for further research in the future. For instance, the close-loop implementation of the considered game (similar to the receding horizon implementation of model predictive control) and the corresponding feedback Nash equilibria are worthwhile to study. How the obtained results in this paper can be used in the mechanism design problems should also be investigated in the future. For instance, in this paper we have not discussed how the reference policy $R_t^{ij}$ should be chosen by the TSO. 

\section*{Acknowledgement}
The authors would like to thank Mr. Matthew T. Morris at the University of Texas at Austin for his contributions to the numerical study in Section~\ref{secsimulation}.

\appendix
\subsection{Proof of Lemma~\ref{lemlimpi}}
\label{app1}

Let $K^i_{N,-n,t}$ denote the number of agents, except agent $n$, which are located at intersection $i$ at time $t$ and let $K^{ij}_{N,-n,t}$ denote the number of agents, except agent $n$, which are located at intersection $i$ at time  $t$ and select intersection $j$ as their next destination. Thus, we have  $K^i_{N,-n,t}=\sum_{l\neq n}\I{i_{l,t}=i}$ and $K^{ij}_{N,-n,t}=\sum_{l\neq n}\I{i_{l,t}=i, j_{l,t}=j}$,
where $\I{\cdot}$ is the indicator function.
Then, $\Pi_{N,n,t}^{ij*}$ can be written as 
\begin{align}
\Pi_{N,n,t}^{ij*}=&\ES{\log\tfrac{1+K^{ij}_{N,-n,t}}{1+K^i_{N,-n,t}}}-\log R_t^{ij}\nonumber\\
=&\ES{\logp{ \tfrac{1+K^{ij}_{N,-n,t}}{N}}}-\nonumber\\
&\ES{\logp{\tfrac{1+K^i_{N,-n,t}}{N}}}-\log R_t^{ij}\nonumber
\end{align}
Using Jensen inequality, we have 
\begin{align}
 \ES{\logp{\tfrac{1+K^{ij}_{N,-n,t}}{N}}}&\leq \logp{\frac{1}{N}+\ES{\tfrac{K^{ij}_{N,-n,t}}{N}}}\nonumber\\
 &\stackrel{(a)}{=}\logp{\frac{1}{N}+\frac{N-1}{N}P_t^{i*}Q_t^{ij*}}\nonumber
\end{align}
where $(a)$ follows from $\ES{\frac{K^{ij}_{N,-n,t}}{N}}=\frac{N-1}{N}P_t^{i*}Q_t^{ij*}$ and the fact that all the agents employ the policy $\left\{Q^{*}_t\right\}$. Thus, we have 
\begin{align}
\limsup_{N\rightarrow\infty}\ES{\log \tfrac{1+K^{ij}_{N,-n,t}}{N}}\leq \log P_t^{i*}Q_t^{ij*}
\end{align}
Next we show the other direction. For $\epsilon\in\left(0,\right.\left.\frac{ P_t^{i*}Q_t^{ij*}}{2}\right]$, we can write $\ES{\log \frac{1+K^{ij}_{N,-n,t}}{N}}$ as
\begin{align}
\ES{\log \tfrac{1+K^{ij}_{N,-n,t}}{N}}=&\ES{\log \tfrac{1+K^{ij}_{N,-n,t}}{N}\I{\tfrac{K^{ij}_{N,-n,t}}{N}>\epsilon}}+\nonumber\\
&\ES{\log \tfrac{1+K^{ij}_{N,-n,t}}{N}\I{\tfrac{K^{ij}_{N,-n,t}}{N}\leq \epsilon}}\nonumber
\end{align}

Using the Hoeffding inequality, it follows that $\frac{K^{ij}_{N,-n,t}}{N}$ converges to $P_t^{i*}Q_t^{ij*}$ in probability as $N$ becomes large. From continues mapping theorem, we have the convergence  of $\log \frac{K^{ij}_{N,-n,t}}{N}$ in probability to $\log P_t^{i*}Q_t^{ij*} $ for $P_t^{i*}Q_t^{ij*}>0$. Similarly, $\I{\frac{K^{ij}_{N,-n,t}}{N}>\epsilon}$ converges to 1 in probability. Thus, from Slutsky's Theorem, we have $\log \frac{1+K^{ij}_{N,-n,t}}{N}\I{\frac{K^{ij}_{N,-n,t}}{N}>\epsilon}$ converges to $P_t^{i*}Q_t^{ij*}$ in distribution. Using Fatou's lemma and the fact that $\log \frac{1+K^{ij}_{N,-n,t}}{N}\I{\frac{K^{ij}_{N,-n,t}}{N}>\epsilon}\geq \log\epsilon $ , we have 
\begin{align}
\liminf_{N\rightarrow\infty} \ES{\log\tfrac{1+K^{ij}_{N,-n,t}}{N}\I{\tfrac{K^{ij}_{N,-n,t}}{N}>\epsilon}}\geq \log P_t^{i*}Q_t^{ij*}\nonumber
\end{align}
We also have 
\begin{align*}
&\abs{\ES{\log \tfrac{1+K^{ij}_{N,-n,t}}{N}\I{\tfrac{K^{ij}_{N,-n,t}}{N}\leq \epsilon}}} \\
&\leq \logp{N}\PRP{\tfrac{K^{ij}_{N,-n,t}}{N}\leq \epsilon}
\end{align*}
Using the Hoeffding inequality, it is straightforward to show that $\PRP{\frac{K^{ij}_{N,-n,t}}{N}\leq \epsilon}$ decays to zero exponentially in $N$ which implies that 
\begin{align}
\lim_{N\rightarrow\infty}\ES{\log \tfrac{1+K^{ij}_{N,-n,t}}{N}\I{\tfrac{K^{ij}_{N,-n,t}}{N}\leq \epsilon}}=0\nonumber
\end{align}
Thus, we have 
\begin{align}
\liminf_{N\rightarrow\infty}\ES{\log \tfrac{1+K^{ij}_{N,-n,t}}{N}}\geq \log P_t^{i*}Q_t^{ij*}
\end{align}
which implies that $\lim_{N\rightarrow\infty}\ES{\logp{ \frac{1+K^{ij}_{N,-n,t}}{N}}}=\log P^{i*}_tQ^{ij*}_t$. Following similar steps, it is straightforward to show that $\lim_{N\rightarrow\infty}\ES{\logp{ \frac{1+K^{i}_{N,-n,t}}{N}}}=\log P^{i*}_t$ which completes the proof.


\bibliographystyle{IEEEtran}
\bibliography{Refs}

\end{document}